\documentclass[12pt,a4paper]{amsart}

\usepackage{amssymb,amsfonts,amsthm,a4,amsmath}

\newtheorem{thm}{Theorem}[section]
\newtheorem{cor}[thm]{Corollary}
\newtheorem{lem}[thm]{Lemma}
\newtheorem{prop}[thm]{Proposition}
\newtheorem{fact}[thm]{Fact}
\newtheorem{cla}[thm]{Claim}
\theoremstyle{definition}
\newtheorem{defn}[thm]{Definition}
\newtheorem{que}[thm]{Question}
\newtheorem{exe}[thm]{Example}
\newtheorem{prob}[thm]{Problem}
\newtheorem{rem}[thm]{Remark}
\numberwithin{equation}{section}

\begin{document}
\title{Sofic profile and computability of Cremona groups}
\author{Yves Cornulier}%
\address{Laboratoire de Math\'ematiques\\
B\^atiment 425, Universit\'e Paris-Sud 11\\
91405 Orsay\\FRANCE}
\email{yves.cornulier@math.u-psud.fr}


\date{\today}

\maketitle

\begin{abstract}
In this paper, we show that Cremona groups are sofic. We actually introduce a quantitative notion of soficity, called sofic profile, and show that the group of birational transformations of a $d$-dimensional variety has sofic profile at most polynomial of degree~$d$. We also observe that finitely generated subgroups of the Cremona group have a solvable word problem. This provides examples of finitely generated groups with no embeddings into any Cremona group, answering a question of S.~Cantat.
\end{abstract}

\section{Introduction}

Let $K$ be a field. The {\em Cremona group} $\textnormal{Cr}_d(K)$ of $K$ in dimension $d$ is defined as the group of birational transformations of the $d$-dimensional $K$-affine space. It can also be described as the group of $K$-automorphisms of the field of rational functions $K(t_1,\dots,t_d)$.

We are far from a global understanding of finitely generated subgroups of Cremona groups. They include, notably, linear groups (since we have an obvious inclusion $\textnormal{GL}_d(K)\subset\textnormal{Cr}_d(K)$), as well as examples of groups that are not linear over any field \cite{CD}. On the other hand, very few restrictions are known about these groups. In the case of $d=2$, and sometimes assuming that $K$ has characteristic zero, there has been a lot of recent progress including \cite{Be,BBe,B2,B3,B4,DI2,DI1,DI3}, see notably the survey \cite{Se2} about finite subgroups, and \cite{B3,BD,BD2,Ca,De} for other subgroups. For $d=3$ there is much less information currently known, in this direction, see \cite{Pr1,Pr2,PS} concerning finite subgroups. For greater $d$, very little information is known; interesting methods have very recently been developed in \cite{Ca2}.

We here provide the following.

\begin{thm}\label{main}
The Cremona group $\textnormal{Cr}_d(K)$ is sofic for all $d$ and all fields $K$. More generally, for any absolutely irreducible variety $X$ over a field $K$, the group of birational transformations $\textnormal{Bir}_K(X)$ is sofic.
\end{thm}

Denoting by $\mathbf{N}$ the set of positive integers, recall that a group $\Gamma$ is {\em sofic} if it satisfies the following: for every finite subset $E$ of $\Gamma$ and every $\varepsilon>0$, there exists $n\in\mathbf{N}$ and a mapping $\phi:E\to \textnormal{Sym}_n$ satisfying
\begin{itemize}
\item $\mathsf{d}_\textnormal{Ham}^n(\phi(g)\phi(h),\phi(gh))\le \varepsilon$ for all $g,h\in E$ such that $gh\in E$;
\item $\phi(1)=1$;
\item $\mathsf{d}_\textnormal{Ham}^n(\phi(u),\phi(v))\ge 1-\varepsilon$ for all $u\neq v$,
\end{itemize}
where $\mathsf{d}_\textnormal{Ham}^n$ is the normalized Hamming distance on the symmetric group $\textnormal{Sym}_n$:
\begin{equation}\label{dham}\mathsf{d}_\textnormal{Ham}^n(u,v)=\frac1n\#\{i:u(i)\neq v(i)\}.\end{equation}

Note that a group is sofic if and only if all its finitely generated subgroups are sofic.
Sofic groups were independently introduced by B.~Weiss \cite{Wei} and Gromov \cite{Gro}. Sofic groups notably include residually finite groups and amenable groups. For more on this topic, see also \cite{ES2,Pe}.

Soficity is a very weak way of approximating a group by finite groups.  Theorem \ref{main} was only known for $n=1$ since then $\textnormal{Cr}_1(K)=\textnormal{PGL}_2(K)$ has all its finitely generated subgroups residually finite. There exists no example, at this time, of a group failing to be sofic, although it is likely to exist. 

Nevertheless, the sofic property is interesting because of its various positive consequences. For instance, if $G$ is a group
and $K$ is a field, a conjecture by Kaplansky asserts that the group algebra $K[G]$ is directly finite, i.e.~satisfies $xy=1\Rightarrow yx=1$. This conjecture is known to hold when $G$ is sofic, by a result of Elek and Szabo \cite{ES1}. Also another conjecture, by Gottschalk, is that if $M$ is a finite set, any $G$-equivariant continuous injective map $M^G\to M^G$ is surjective (the product $M^G$ being endowed with the product topology, which makes it a compact topological space); Gromov \cite{Gro} proved that this is true when $G$ is sofic.

The second restriction, of a totally different nature, is the following.

\begin{thm}\label{cswp}
For every field $K$ and integer $d\ge 0$, every finitely generated subgroup of $\textnormal{Cr}_d(K)$ has a solvable word problem.
\end{thm}
To avoid any reference to group presentations, here we define a group to have a solvable word problem if it is either finite or isomorphic to the set $\mathbf{N}$ endowed with a recursive group law, see \S\ref{s:wp}.

This provides explicit examples of finitely generated --~or even finitely presented~-- groups that are not subgroups of any Cremona group. (This answers a question of S.~Cantat.) 

\begin{exe}\label{zwz}
Let $I$ be a subset of $\mathbf{N}$. If the group 
$$G_I=\langle t,x\mid [t^nxt^{-n},x]=1,\;\forall n\in I\rangle,\qquad(\textnormal{where }[g,h]=ghg^{-1}h^{-1})$$
has a solvable word problem then $I$ is recursive. Indeed, an elementary argument shows that for $n\in\mathbf{N}$, we have $[t^nxt^{-n},x]=1$ in $G_I$ if and only if $n\in I$, i.e.\ $([t^nxt^{-n},x])_{n\in\mathbf{N}}$ is an independent family of relators \cite{B61}. (It can be shown that, conversely, if $I$ is recursive then $G_I$ has a solvable word problem, but this is irrelevant here.) Thus by Theorem \ref{fgcsw}, if $I$ is not recursive then $G_I$ does not embed into any Cremona group (if $I$ is recursively enumerable then note that $G_I$ is recursively presented).

Construction of {\em finitely presented} groups with an unsolvable word problem is considerably harder, and was done by Boone and Novikov. It follows from 
Theorem \ref{fgcsw} that these groups do not embed into any Cremona group. Mark Sapir indicated to me that there exist, on the other hand, finitely presented groups, constructed in \cite{BRS} whose word problem is solvable, but not in exponential time. Thus these groups do not embed into Cremona groups although they have a solvable word problem.
\end{exe}

On the other hand, Miller III \cite{miller} improved the construction of Boone and Novikov by exhibiting nontrivial finitely presented groups all of whose nontrivial quotients have a non-solvable word problem. We deduce the following corollary.

\begin{cor}
There exists a nontrivial finitely presented group with no non-trivial homomorphism to any Cremona group over any field.
\end{cor}

However, Cantat's problem is in no way closed, as we are still far from even a rough understanding of the structure of subgroups of Cremona groups.
Many natural instances of groups have an efficiently solvable word problem, and yet are not expected to embed into any Cremona group, e.g, when they fail to satisfy the Tits Alternative (which holds in $\textnormal{Cr}_2(\mathbf{C})$ by a result of Cantat \cite{Ca2}). For example, it is expected that if $n(d)$ is the smallest number such that $\textnormal{Cr}_{n(d)}$ contains a copy of the symmetric group on $d$ letters, then $\lim_{d\to\infty}n(d)=\infty$. This would imply in particular that the group of finitely supported permutations of the integers (or any larger group) does not embed into any Cremona group.

Theorem \ref{main} is proved in Section \ref{secmainproof} in the case of Cremona groups, and in general in Section \ref{gv}. Although the latter supersedes the former, the proof in the Cremona case is much less technical, so we include it. The main two steps are
\begin{enumerate}
\item Reduction to finite fields;
\item case of finite fields.
\end{enumerate}
The second step uses the ``quasi-action" on the set of points, using that the indeterminacy set being of positive codimension, its number of points over a given finite field can be bounded above in a quantitative way. The first step is fairly easy in the case of Cremona groups, and is much more technical in the general case.

No example is known of a non-sofic group; in particular, so far Theorem~\ref{main} provides no example of groups that cannot be embedded into any Cremona group. However, the proof provides a property stronger than soficity, namely that $\textnormal{Cr}_d(K)$ (or more generally $\textnormal{Bir}_K(X)$ when $X$ is $d$-dimensional) has its ``{\em sofic profile}" in $O(n^d)$ (see Corollary \ref{mainc}). This might result in new explicit examples of groups not embedding into Cremona groups, without exhibiting non-sofic groups, and with an efficiently solvable word problem. See Section \ref{nnsofic}, in which the sofic profile is defined, and related to the classical isoperimetric profile (or F\o lner function).

{\bf Outline of the paper.} Section \ref{secmainproof} contains the proof of soficity of the Cremona group $\textnormal{Cr}_d(K)$. Section \ref{nnsofic} introduces the notion of sofic profile, yielding various examples. Then Section \ref{gv} proves Theorem \ref{main} in full generality; although the proof uses only basic commutative algebra that are extensively used by algebraic geometers (generic flatness, openness conditions), these notions are not of the utmost common background for readers in geometric group theory, who can stick to Section \ref{secmainproof} and \ref{nnsofic}. Section \ref{nnsofic} can also be read independently, without reference to Cremona groups.

Finally, Section \ref{s:wp}, which is also independent of the remainder, includes a proof of Theorem \ref{cswp}, as well as related remarks.

We end this introduction by the following open question:

\begin{que}
For $d\ge 2$, and any field $K$, is $\textnormal{Cr}_d(K)$ locally residually finite (i.e., is every finitely generated subgroup residually finite)? approximable by finite groups (see Definition \ref{dapprox})? (I heard the question of local residual finiteness for $\textnormal{Cr}_d(\mathbf{C})$ from S.~Cantat.)
\end{que}

\medskip

\noindent{\bf Acknowledgements.} I thank Jeremy Blanc for pointing out several inaccuracies in an earlier version of the paper. I am grateful to Serge Cantat for stimulating discussions. I thank Julie Deserti and the referee for many useful corrections. I also thank Goulnara Arzhantseva and Pierre-Alain Cherix for letting me know about their work.

\setcounter{tocdepth}{1}
\tableofcontents

\section{Soficity of Cremona groups}\label{secmainproof}

We begin by the notion of approximation, studied in a much wider context by Malcev in \cite{Ma} and classical in model theory.

\begin{defn}\label{dapprox}
Let $\mathcal{C}$ be a class of groups. We say that a group $G$ is {\em approximable by the class $\mathcal{C}$} (or {\em initially sub-$\mathcal{C}$} in Gromov's terminology \cite{Gro}) if for every finite symmetric subset $S$ of $G$ containing 1, there exists a group $H\in\mathcal{C}$ and an abstract injective map $\phi:S\to H$ such that $\phi(1)=1$ and for all $x,y,z\in S$ we have $\phi(x)\phi(y)=\phi(z)$ whenever $xy=z$ (in particular $\phi(x^{-1})=\phi(x)^{-1}$ for all $x\in S$). Equivalently, $G$ is approximable by the class $\mathcal{C}$ if and only if it is isomorphic to a subgroup of an ultraproduct of groups of the class $\mathcal{C}$.
\end{defn}

Note that plainly, if a group is approximable by $\mathcal{C}$ then so are all its subgroups, and conversely if all its finitely generated subgroups are approximable by $\mathcal{C}$, then so is the whole group. 

 It is straightforward from the definition that if a group is approximable by sofic groups, then it is sofic as well. Therefore the first part of Theorem \ref{main} follows from the following two propositions.

\begin{prop}\label{creap}
For any field $K$ and $d$, the Cremona group $\textnormal{Cr}_d(K)$ is approximable by the family $$\{\textnormal{Cr}_d(\mathbf{F}):\; \mathbf{F}\textnormal{ finite field}\}.$$
\end{prop}

\begin{prop}\label{crefinisofic}
For any finite field $\mathbf{F}$ and $d$, the Cremona group $\textnormal{Cr}_d(\mathbf{F})$ is sofic.
\end{prop}

\begin{rem}
A strengthening of Proposition \ref{creap} would be the assertion that for every field $K$, the group $\textnormal{Cr}_d(K)$ is ``locally residually $\textnormal{Cr}_d$ of a finite field", in the sense that every finitely generated subgroup embeds into a product of groups of the form $\textnormal{Cr}_d(\mathbf{F})$ with $\mathbf{F}$ finite field; we do not know if this assertion holds. On the other hand, it is clear that every finitely generated subgroup of $\textnormal{Cr}_d(K)$ is contained in $\textnormal{Cr}_d(L)$ for some finitely generated subfield $L$ of $K$.
\end{rem}

To prove the propositions, we begin by some basic material about birational transformations of affine spaces. 
Consider $f=(f_1,\dots,f_d)$, where $f_i\in K(t_1,\dots,t_d)$. Its (affine) {\em  indeterminacy\footnote{The notion of indeterminacy set is sensitive to our choice to work in affine coordinates; here the indeterminacy set usually has codimension 1, while in projective coordinates the indeterminacy set has codimension at least~2.} set} $X_f$ is by definition the union of the zero sets of the denominators of the $f_i$ (written in irreducible form).
To such a $d$-tuple corresponds to the regular map defined outside its singular set mapping, for any extension $L$ of $K$ \[(x_1,\dots,x_d)\in L^d\smallsetminus X_f(L)\quad\textnormal{to}\quad(f_1(x_1,\dots,x_d),\dots,f_d(x_1,\dots,x_d)).\] We say that $f$ is non-degenerate if $f$ has a Zariski-dense image. If $g$ is another $d$-tuple and $f$ is non-degenerate, we can define the composition $g\circ f$ by
 $$\Big(g_1\big(f_1(t_1,\dots,t_d),\dots,f_d(t_1,\dots,t_d)\big),\dots ,g_d(\dots)\Big)\in K(t_1,\dots,t_d).$$
The non-degenerate $d$-tuples thus form a semigroup under composition, and by definition the Cremona group $\textnormal{Cr}_d(K)$ is the set of invertible elements of this semigroup.
If $f\in\textnormal{Cr}_d(K)$ and $f'$ is its inverse (which will be written $f^{-1}$ in the sequel, but not in the next line in order to avoid a confusion with the inverse image by the map $f$ defined outside $X_f$), we define the {\em singular set} $Z_f=X_f\cup f^{-1}(X_{f'})$. Then $f$ induces a bijection, for every extension $L$ of $K$
\[L^d\smallsetminus Z_f\to L^d\smallsetminus Z_{f^{-1}}.\]

\begin{proof}[Proof of Proposition \ref{creap}]Since any field extension $K\subset L$ induces a group embedding $\textnormal{Cr}_d(K)\subset\textnormal{Cr}_d(L)$, it is enough to prove the proposition when $K$ is algebraically closed.

Let $W$ be a finite symmetric subset of $\textnormal{Cr}_d(K)$ containing 1. Write each coordinate of every element of $W$ as a quotient of two polynomials. Let $c_1$ be the product in $K$ of all nonzero coefficients of denominators of coordinates of elements of $WW$; let $c_2$ be the product of all nonzero coefficients of numerators of coordinates of elements of the form $u-v$, when $(u,v)$ ranges over pairs of distinct elements of $W$. Let $A$ be the domain generated by all coefficients of elements of $W$, so $c=c_1c_2\in A-\{0\}$. Since the ring $A$ is residually a finite field \cite{Mal}, there exists a finite quotient field $\mathbf{F}$ of $A$ in which $\bar{c}\neq 0$, where $x\mapsto \bar{x}$ is the natural projection $A\to \mathbf{F}$. If $u\in \mathbf{F}$, we can view $u$ as an element of $\mathbf{F}(t_1,\dots,t_d)^d$ as above (the denominator does not vanish because $\bar{c}_1\neq 0$). Also, the condition $\bar{c}_1\neq 0$ implies that whenever $uv=w$, we also have $\bar{u}\bar{v}=\bar{w}$. In particular, since $W$ is symmetric and contains 1, it follows that the elements $\bar{u}$ are invertible, i.e.\ belong to $\textnormal{Cr}_d(\mathbf{F})$. Finally, whenever $u\neq v$, since $\bar{c}_2\neq 0$, we have $\bar{u}\neq\bar{v}$.
\end{proof}

\begin{rem}
It follows from the proof that $\textnormal{Cr}_d(K)$ is approximable by some suitable subclasses of the class of $d$-Cremona groups over finite fields: if $K$ has characteristic $p$ it is enough to restrict to finite fields of characteristic $p$, and if $K$ has characteristic $0$ it is enough to restrict to the class of finite fields of characteristic $p\ge p_0$ for any fixed $p_0$. Also, if $K=\mathbf{Q}$, it is enough to restrict to the class of cyclic fields $\mathbf{Z}/p\mathbf{Z}$ (for $p\ge p_0$). 
\end{rem}

\begin{proof}[Proof of Proposition \ref{crefinisofic}]
Write $\mathbf{F}=\mathbf{F}_q$. Let $W$ be a finite symmetric subset of $\textnormal{Cr}_d(\mathbf{F}_q)$ containing 1.

For any $u\in\textnormal{Cr}_d(\mathbf{F}_q)$ and for every $\mathbf{F}_q$-field $L$, $u$ induces a bijection from $L^d-Z_u$ to $L^d-Z_{u'}$. We extend it arbitrarily (for each given $L$) to a permutation $\hat{u}$ of $L^d$.

Note that for all $u,v$, the permutations $\hat{u}\hat{v}$ and $\widehat{uv}$ coincide on the complement of $Z_v\cup v^{-1}(Z_u)$. 

Then there exists a constant $C>0$ such that for all $u\in W$ and all $m$ we have $\#Z_u(\mathbf{F}_{q^m})\le Cq^{m(d-1)}$ (this is a standard consequence, for instance, of the Lang-Weil estimates \cite{LW} but can be checked directly). 

So, when $L=\mathbf{F}_{q^m}$ the Hamming distance in $\textnormal{Sym}(L^d)$ between $\hat{u}\hat{v}$ and $\widehat{uv}$ is $\le 2Cq^{-m}$, which tends to $0$ when $m$ tends to~$+\infty$.

Also, by considering the zero set $D_{uv}$ of the numerator of $u-v$, we obtain that if $u\neq v$, the Hamming distance from $\hat{u}$ and $\hat{v}$ is $\ge 1-2C'q^{-m}$, for some fixed constant $C'$ and for all $m$.
We thus proved that $\textnormal{Cr}_d(K)$ is sofic.\end{proof}

\begin{rem}\label{nsofic}
We actually proved that for every field $K$, the group $\textnormal{Cr}_d(K)$ satisfies the following property: for every finite subset $S\subset\textnormal{Cr}_d(K)$ there is a constant $c_S>0$ such that for every integer $n$ there exists $k\le n$ and a map $S\to\textnormal{Sym}_k$ satisfying 

\begin{itemize}
\item $\mathsf{d}_\textnormal{Ham}^k(\phi(g)\phi(h),\phi(gh))\le c_Sn^{-1/d}$ for all $g,h\in S$ such that $gh\in S$;
\item $\phi(1)=1$ and $\phi(g)=\phi(g^{-1})$ for all $g\in S$;
\item $\mathsf{d}_\textnormal{Ham}^k(\phi(u),\phi(v))\ge 1-c_Sn^{-1/d}$ for all $u\neq v$.
\end{itemize}
where $\mathsf{d}_\textnormal{Ham}^k$ is the normalized Hamming distance on the symmetric group $\textnormal{Sym}_k$. (In Section \ref{nnsofic}, we will interpret this by saying that the ``sofic profile" of $\textnormal{Cr}_d(K)$ is in $O(n^d)$.) 
Note that for every integer $m\ge 1$ there exists a distance-preserving homomorphism $(\textnormal{Sym}_k,\mathsf{d}_\textnormal{Ham}^k)\to(\textnormal{Sym}_{mk},\mathsf{d}_\textnormal{Ham}^{mk})$; in particular $k$ can be chosen so that $k\ge n/2$.
\end{rem}

\section{Sofic profile}\label{nnsofic}

\subsection{Isoperimetric profile}
Let us first recall the classical notion of isoperimetric profile (or F\o lner function) of a group $G$ (see \cite{PiS} for a more detailed survey). If $S,X$ are subsets of $G$, define $\partial_S X=SX-X$. Following Vershik \cite{V}, define the isoperimetric profile of $(G,S)$ as the nondecreasing function $\alpha_{G,S}$ defined for $r>1$ by
$$\alpha_{G,S}(r)=\inf\{n\ge 1:\exists E\subset G,\#(E)=n,\#(\partial_S(E))/\#(E)< r^{-1}\},$$
where $\inf\emptyset=+\infty$. The group $G$ is {\em amenable} if $\alpha_{G,S}(r)<+\infty$ for every finite subset $S\subset G$ and all $r>1$. The equivalence between this definition and the original definition of amenability by von Neumann \cite{vNe} is due to F\o lner \cite{Fol}.

Note that the isoperimetric profile of $(G,S)$ is bounded for every finite subset $S$ (i.e., $\forall S$ finite, $\sup_r\alpha_{G,S}(r)<\infty$) if and only if $G$ is locally finite.

A convenient fact is that the asymptotics of $\alpha_{G,S}$ does not depend on $S$, when the latter is assumed to be a symmetric generating subset of $G$.

If $u,v:\mathopen]1,\infty\mathclose[\to [0,\infty]$ are nondecreasing functions, we write $u\preceq v$ if there exist positive real constants such that $u(r)\le Cv(C'r)+C''$ for all $r\ge 1$, and we write $u\simeq v$ if $u\preceq v\preceq u$. 

\begin{rem}\label{ipfg}
If $G$ is a finitely generated group and $S,T$ are finite subsets, with $S$ a symmetric generating subset, then $\alpha_{G,S}\succeq \alpha_{G,T}$. In particular, if $T$ is also a symmetric generating subset then $\alpha_{G,S}\simeq \alpha_{G,T}$. Thus if $G$ is finitely generated, the $\simeq$-class of the function $\alpha_{G,S}$ does not depend on the finite symmetric generating subset $S$. It is usually called the {\em isoperimetric profile} of~$G$ (and is $\simeq\infty$ if and only if $G$ is non-amenable).
\end{rem}

By a result of Coulhon and Saloff-Coste \cite{CS}, the isoperimetric profile grows at least as fast as the volume growth.

If $G=\mathbf{Z}^d$, the isoperimetric profile is $\simeq r^{d}$ and this is optimal; the same estimate holds for groups of polynomial growth of degree $d$. If $G$ has exponential growth, then the isoperimetric profile is $\succeq \exp(r)$ and this is optimal for polycyclic groups~\cite{Pit}.

Let us mention that the isoperimetric profile is closely related to the non-increasing function $I_{G,S}$ (also called ``isoperimetric profile" in some papers) defined by
\[I_{G,S}(n)=\inf\{\#(\partial_S(E))/\#(E):\;E\subset G,\;0<\#(E)\le n \}.\]
We check immediately that for all reals $r\ge 1$ and integers $n\ge 1$ we have $\alpha_{G,S}(r)\le n$ $\Leftrightarrow$ $r<I_{G,S}(n)^{-1}$. Thus $\alpha_{G,S}$ and $1/I_{G,S}$ are essentially inverse functions to each other. For instance, if $G$ is a polycyclic group of exponential growth then $I_{G,S}$ grows as $1/\log(n)$ whenever $S$ is a finite symmetric generating subset.

\subsection{Sofic profile and basic properties}

Here we introduce a notion of {\em sofic profile}, intuitively associated to a group, but more formally associated to its finite pieces, or ``chunks". A similar, but different notion of ``sofic dimension growth" of a finitely generated group was independently introduced by Arzhantseva and Cherix (see Remark \ref{soprac} for the precise definition and comments). 

\begin{defn}Let us call {\em chunk} a finite set $E$, endowed with a basepoint $1_E$ and a subset $D$ of $E\times E\times E$ satisfying the condition $(x,y,z),(x,y,z')\in D$ implies $z=z'$. So we can view it as a partially defined composition law $(x,y)\mapsto z$ and we write $xy=z$ to mean that $(x,y,z)\in D$. 

If $E$ is an abstract chunk and $G$ is a group, we call {\em representation} of $E$ into $G$ a mapping $f:E\to G$ such that $f(1_E)=1_G$ and $f(x)f(y)=f(z)$ whenever $xy=z$.
\end{defn}

If $E$ is a subset of a group $G$ with $1_G\in E$, it is naturally a chunk with basepoint $1_G$ by setting $xy=z$ whenever this holds in the group $G$. We call it a chunk of $G$ (symmetric chunk if $E$ is symmetric in $G$).

This allows the following immediate restatement of the notion of approximability from Definition \ref{dapprox}.

\begin{fact}
Let $\mathcal{C}$ be a class of groups. Then a group $G$ is approximable by the class $\mathcal{C}$ if and only if every chunk of $G$ has an injective representation into a group in the class~$\mathcal{C}$.\qed
\end{fact}

\begin{defn}
Let $E$ be a chunk. If $n$ is an integer and $\varepsilon>0$, define an $\varepsilon$-{\em morphism} from $E$ to $\textnormal{Sym}_n$ to be a mapping $f:E\to \textnormal{Sym}_n$ such that $f(1_E)=\textnormal{id}$ and $\mathsf{d}_\textnormal{Ham}^n(f(xy),f(x)f(y))\le\varepsilon$ for all $x,y\in E$, where the Hamming distance $\mathsf{d}_\textnormal{Ham}^n$ is defined in (\ref{dham}). A mapping from $E$ to the symmetric group $\textnormal{Sym}_n$ is said to be $(1-\varepsilon)$-expansive if $\mathsf{d}_\textnormal{Ham}^n(x,y)\ge 1-\varepsilon$ whenever $x,y$ are distinct points of $E$.

Define the {\em sofic profile} of the chunk $E$ as the non-decreasing function
\begin{align*}\sigma_E(r)= & \inf\big\{n:\;\exists f:E\to(\textnormal{Sym}_n,\mathsf{d}_\textnormal{Ham}^n),\\\ & \qquad\qquad f\textnormal{ is a }(1-r^{-1})\textnormal{-expansive }r^{-1}\textnormal{-morphism}\big\}\qquad(r> 1),\end{align*}
where $\inf\emptyset=+\infty$.
Say that the chunk $E$ is {\em sofic} if its sofic profile takes finite values: $\sigma_E(r)<\infty$ for all $r\ge 1$. 
\end{defn}

The following elementary fact shows that the sofic profile of a chunk is either bounded or grows at least linearly.

\begin{fact}\label{spb}
If $E$ is a chunk, we have the alternative:
\begin{itemize}
\item either $E$ has an injective representation into a finite group and hence its sofic profile is bounded, i.e.\ $\sup_r\sigma_E(r)<\infty$;
\item or its sofic profile satisfies $\sigma_E(r)\ge r$ for all $r>1$.
\end{itemize}
\end{fact}
\begin{proof}
If $E$ has an injective representation into a finite group $H$, then this representation is a $(1-r^{-1})$-expansive $r^{-1}$-morphism for every $r> 1$. So, picking $n$ such that $H$ embeds into $\textnormal{Sym}_n$, we have $\sigma_E(r)\le n$ for all $r\ge 1$.

To show the alternative, assume that the second condition fails, namely $\sigma_E(r)<r$ for some $r>1$. So $E$ has a $(1-r^{-1})$-expansive $r^{-1}$-morphism $\phi$ into $\textnormal{Sym}_n$ for some $n<r$; since $r>1$, necessarily $\phi$ is injective. Since the Hamming distance $\mathsf{d}_\textnormal{Ham}^n$ takes values in $\{0,1/n,\dots,1\}$ and $r^{-1}<n^{-1}$, this shows that $\phi$ is a 0-morphism, i.e.\ is an injective representation.
\end{proof}

\begin{defn}
A group $G$ is {\em sofic} if every chunk in $G$ is sofic, i.e.\ $\sigma_E(r)<\infty$ for every chunk $E$ in $G$ and~$r\ge 1$.
\end{defn}

This is a restatement of the definition given in the introduction. We wish to attach to $G$ a ``sofic profile", namely the family of the function $\sigma_E$, when $E$ ranges over finite subsets of $G$. Let us be more precise.

\begin{defn}\label{spg} The {\em sofic profile} of $G$ is the family of $\simeq$-equivalence classes of the functions $\sigma_E$ when $E$ ranges over finite subsets of $G$. If this class has greatest element (in the set of classes of nondecreasing functions modulo $\simeq$), namely the class of a (unique up to $\simeq$) function $u$, we say that the sofic profile of $G$ is $\simeq u$.
\end{defn}

We have the following immediate consequence of Fact \ref{spb}:
\begin{fact}\label{gap}
Let $G$ be a group. We have the alternative:
\begin{itemize}
\item $G$ has a bounded sofic profile, in the sense that $\sup_r\sigma_E(r)<\infty$ for every chunk $E$ in $G$; this occurs precisely when $G$ is approximable by finite groups;
\item or the sofic profile of $G$ grows at least linearly; more precisely there exists a chunk $E$ in $G$ such that $\sigma_E(r)\ge r$ for all $r>1$.\qed
\end{itemize}
\end{fact}

The class of groups approximable by (the class of) finite groups is well-known \cite{St,VG}, and they are also called ``LEF-groups", which stands for ``Locally Embeddable into Finite groups". A residually finite group is always approximable by finite groups, and the converse holds for finitely presented groups, but not for general finitely generated groups (see \cite{St,VG}).

\begin{exe}
Most familiar groups are {\em locally residually finite} (in the sense that every finitely generated subgroup is residually finite). Such groups are approximable by finite groups and hence have a bounded sofic profile. This includes:
\begin{itemize}
\item abelian groups, and more generally abelian-by-nilpotent groups (groups with an abelian normal subgroup such that the quotient is nilpotent) \cite{Hall};
\item linear groups, i.e.\ subgroups of $\textnormal{GL}_n(A)$ for any $n$ and commutative ring $A$ (see \cite{Weh});
\item groups of automorphisms of affine varieties over a field \cite{BL};
\item compact groups (i.e., groups that admit a Hausdorff compact group topology), by the Peter-Weyl theorem;
\end{itemize}

Examples of groups approximable by finite groups are (locally finite)-by-cyclic groups. Indeed, if such a group is finitely generated, it is, by \cite[Theorem A]{BiS}, an inductive limit of a sequence of finitely generated virtually free groups. Such groups are not necessarily locally residually finite \cite{St,VG}. 

For examples of groups not approximable by finite groups, see Examples \ref{bse} and \ref{otherisol}.
\end{exe}

To pursue the discussion, we use the following useful terminology, which in a certain sense allows to think of the sofic profile as a function.

\begin{defn}\label{sofp}
Given fixed functions $u,v$, we say that the sofic profile of $G$ is $\preceq u$ if $\sigma_E\preceq u$ for {\em every} chunk $E$ in $G$ and is $\succeq v$ if $\sigma_E\succeq v$ for {\em some} chunk $E$ in $G$. Similarly we say that the sofic profile of $G$ is {\em polynomial} (resp.\ {\em at most polynomial of degree $k$}) if for every chunk $E$ of $G$, there is a polynomial (resp.\ polynomial of degree $k$) $f$ such that $\sigma_E\preceq f$.
\end{defn}

Note that to say that the sofic profile is at most polynomial of degree 0 just means that it is bounded.

\begin{rem}
An advantage of this definition is that for a group it depends only on its chunks, and therefore, tautologically, if any group in $\mathcal{C}$ has the property that its sofic profile is $\preceq u(r)$, then it still holds for any group approximable by the class $\mathcal{C}$. In particular, for any $u$, to have sofic profile $\preceq u(r)$ is a closed property in the space of marked groups (see e.g.\ \cite[Sec.~1]{CGP} for basics about this space).
\end{rem}

\begin{rem}In contrast to the isoperimetric profile, it is not true that the sofic profile of a finitely generated group $G$ is the sofic profile of any chunk attached to a symmetric generating subset (with unit). A natural assumption is to require that the corresponding subset $S$ contains enough relations, namely that $G$ has a presentation with $S$ as set of generators and relators of length $\le 3$. However, I do not know if for such an $S$, denoting by $E$ the corresponding chunk, $\sigma_E$ is the sofic profile of $G$ in the sense of Definition \ref{spg}, nor if an arbitrary presented group has a sofic profile $\simeq$-equivalent to some function as in Definition \ref{spg}.
\end{rem}

\begin{rem}\label{soprac}
The notion of sofic dimension growth due to Arzhantseva and Cherix (work in progress) is the following. Let $G$ be generated by a finite symmetric subset $S$. The {\em sofic dimension growth} $\phi(n)$ is, in the language introduced here, $\phi_S(n)=\sigma_{S^n}(n)$. Arzhantseva and Cherix show that its asymptotics only depend on $G$ and not on the choice of $S$, and related it to the isoperimetric profile. However, it is quite different in spirit to the sofic profile, because it takes into account the shape of balls. In particular, the sofic dimension growth is bounded only for finite groups.

I am not able to adapt the specification process used to estimate the sofic profile of Cremona groups (Proposition \ref{creap}) to give any upper bound on the sofic dimension growth of their finitely generated subgroups. This is probably doable, but at the cost of some tedious estimates on the degrees of singular subvarieties arising in the proof, which would not give better than an exponential upper bound for the sofic dimension growth.

Note that the knowledge of the function of two variables $\Phi(m,n)=\sigma_{S^m}(n)$ encompasses both the sofic dimension growth $\phi_S(n)=\Phi(n,n)$ and the sofic profile (asymptotic behavior of $\Phi(m,n)$ when $m$ is fixed). 
\end{rem}

\subsection{Sofic vs isoperimetric profile}

Informally, soficity of $G$ means that points in $G$ are well separated by ``quasi-actions" of $G$ on finite sets, and amenability is the additional requirement that these finite sets lie inside $G$ with the action by the left multiplication. With this in mind, it is elementary to check that the sofic profile is asymptotically bounded above by the isoperimetric profile; precisely we have the following result.

\begin{prop}
For any finite subset $S$ of $G$, we have the following comparison between the sofic profile and the isoperimetric profile
\[\sigma_S(r/3)\le \alpha_{G,S}(r),\quad\forall r\ge 3.\]
\end{prop}
\begin{proof}
Suppose that $\alpha_{G,S}(r)\le n$ and let us show that $\sigma_S(r/3)\le n$. By assumption there exists $E\subset G$ with $0<\#(E)\le n$ and $\#(SE-E)/\#(E)<r^{-1}$. For $s\in S$, define $\phi(s):E\to E$ to map $x\mapsto sx$ if $sx\in E$, and extend it arbitrarily to a bijection. By assumption, for each $s$, the proportion of $x\in E$ such that $\phi(s)(x)=sx$ is $>1-r^{-1}$. It follows that the Hamming distance of $\phi(s)$ and $\phi(s')$ is $>1-2r^{-1}$ whenever $s,s'\in S$ and $s\neq s'$, and the Hamming distance between $\phi(st)$ and $\phi(s)\phi(t)$ is $<3r^{-1}$ whenever $s,t,st\in S$. So $\sigma_{S}(r/3)\le n$.  
\end{proof}

It is known \cite{ES2} that any sofic-by-amenable group (i.e.~lying in an extension with sofic kernel and amenable quotient) is still sofic. The proof given there is an explicit construction, yielding without any change the following.

\begin{thm}\label{ESq}
Let $G$ be a group in a short exact sequence $1\to N\to G\to Q\to 1$. 
Then for every symmetric chunk $E$ in $G$ there exists a symmetric chunk $E'$ in $N$ and a finite symmetric subset $S$ in $Q$ such that $\sigma_E(r)\le\sigma_{E'}(r)\alpha_{Q,S}(r)$ for all $r>1$.

In particular, given nondecreasing functions $u,v:\mathopen]1,\infty\mathclose[\to[1,\infty]$, if the sofic profile of $N$ is $\preceq u(r)$ and the isoperimetric profile of $Q$ is $\preceq v(r)$, then the sofic profile of $G$ is $\preceq u(r)v(r)$.
\end{thm}

\begin{exe}It follows from Theorem \ref{ESq} that the class of groups with polynomial sofic profile (see Definition \ref{sofp}) is stable under extension with virtually abelian quotients. Since it is also stable under taking filtering inductive limits, it follows that every elementary amenable group has a polynomial sofic profile. (Recall that the class of elementary amenable groups is the smallest class containing the trivial group and stable under direct limits and extensions with finitely generated virtually abelian quotients.) In particular, any solvable group has a polynomial sofic profile. Note that this does not prove that it has a sofic profile $\preceq r^d$ for some $d$, as the degree $d$ may depend on the chunk.
\end{exe}

\begin{exe}\label{bse}
For $k,\ell\in\mathbf{Z}\smallsetminus\{0\}$, the sofic profile of the Baumslag-Solitar group 
$$\Gamma=\textnormal{BS}(k,\ell)=\langle t,x|tx^kt^{-1}\rangle$$
is at most linear (i.e.\ is $\preceq r$); more precisely it is linear (i.e.\ is $\simeq r$), unless $|k|=1$, $|\ell|=1$, or $|k|=|\ell|$, in which case it is bounded.
\end{exe}
\begin{proof}
Let $N$ be the kernel of the homomorphism of $\Gamma$ onto $Q=\mathbf{Z}$ mapping $(t,x)$ to $(1,0)$. The assertion follows from Theorem \ref{ESq} the fact that the isoperimetric profile of $\mathbf{Z}$ is linear, and that $N$ is approximable by finite groups (so its sofic profile is bounded). Let us check the latter fact: using that $\Gamma$ is the HNN-extension of $\mathbf{Z}$ by the two embeddings of $\mathbf{Z}$ into itself by multiplication by $k$ and $\ell$ respectively, the group $N$ is an iterated free product with amalgamation $\cdots\mathbf{Z}\ast_\mathbf{Z}\mathbf{Z}\ast_\mathbf{Z}\mathbf{Z}\ast_\mathbf{Z}\cdots$, where each embedding of $\mathbf{Z}$ to the left, resp.\ to the right, is given by multiplication by $k$, resp.\ by $\ell$ \cite[I.1.4, Prop.\ 6]{Se}. This group is locally residually finite, i.e.\ every such finite iteration $\mathbf{Z}\ast_\mathbf{Z}\mathbf{Z}\ast_\mathbf{Z}\dots\ast_\mathbf{Z}\mathbf{Z}$ is residually finite; this follows, for instance, from \cite{evans}. (In case $k,\ell$ are coprime, R. Campbell \cite{Camp} checked that $N$ itself is not residually finite, and even that all its finite quotients are abelian.)

By Fact \ref{gap}, the sofic profile is $\simeq r$ unless $\Gamma$ is approximable by finite groups. Since $\Gamma$ is finitely presented, 
this occurs if and only if $\Gamma$ is residually finite, which precisely holds in the given cases, by a result of Meskin \cite{meskin} (correcting an error in \cite{BS}). 
\end{proof}

Note that the fact that $\textnormal{BS}(k,\ell)$ is residually solvable (indeed, free-by-metab\-elian) immediately implies its soficity, but yields a much worse upper bound on its sofic profile.

\begin{prob}
Develop methods to compute lower bounds for the sofic profile of explicit groups. Is there any group for which the sofic profile is unbounded and not $\simeq r$? Can such a group be sofic?
\end{prob}

This problem only concerns groups not approximable by finite groups, since otherwise the sofic profile is bounded. Otherwise the sofic profile grows at least linearly as we observed above, but we have no example with a better lower bound.

\begin{exe}\label{otherisol}Here are some examples of finitely generated groups not approximable by finite groups, whose sofic profile could be looked over.
\begin{itemize}
\item Infinite isolated groups. A group $G$ is by definition {\em isolated} if it has a chunk $S$ such that any injective representation of $S$ into a group $H$ extends to an injective homomorphism $G\to H$. (This clearly implies that $G$ is generated by $S$ and actually is presented with the set of conditions $st=u$, $s,t,u\in S$ as a set of relators.) These include finitely presented simple groups. Many more examples are given in \cite{CGP}, e.g.\ Thompson's group $F$ of the interval. It includes several examples that are amenable (solvable or not) and therefore sofic. We can also find in \cite{CGP} examples of non-amenable isolated groups but whether they are sofic is not known; however an example of a non-amenable isolated group that is known to be sofic, is given in \cite{Csofic}.

\item Other finitely presented non-residually finite groups. This includes most Baumslag-Solitar groups as mentioned in Example \ref{bse}, as well as various other one-relator groups \cite{B69,BMT}. Another example is Higman's group \cite[I.1.4, Prop.\ 5]{Se}
$$\langle x_1,x_2,x_3,x_4|\;x_{i-1}x_{i}x_{i-1}^{-1}=x_{i}^2\;(i=1,2,3,4 \mod 4)\rangle,$$
which has no proper subgroup of finite index. Whether it is sofic is not known.

\item Direct products of the above groups. For instance, $\textnormal{BS}(2,3)^d$ has sofic profile $\preceq n^d$.
\end{itemize}
\end{exe}

\section{General varieties}\label{gv}

The purpose of this section is to prove Theorem \ref{main} in its general formulation (for an arbitrary absolutely irreducible variety). Since the group of birational transformations of an absolutely irreducible variety can be canonically identified with that of an open affine subset, we can, in the sequel, stick to affine varieties. 

If $X$ is an affine variety over the field $K$, we define a {\em specification} of $X$ over a finite field $\mathbf{F}$ as an affine variety $X''$ over $\mathbf{F}$ satisfying the following condition. Denoting by $B$ and $B''$ the $K$-algebras of functions of $X$ and the $\mathbf{F}$-algebra of functions on $X''$, there exists a finitely generated subdomain $A$ of $K$, a finitely generated $A$-subalgebra $B'$ of $B$, a surjective homomorphism $A\to \mathbf{F}$, so that $B'\otimes_A \mathbf{F}\simeq B''$ as $A$-algebras, and the natural $K$-algebra homomorphism $B'\otimes_A K\to B$ is an isomorphism. Note that $\dim(X'')\le\dim(X)$.

\begin{prop}\label{redfinig}
Let $X$ be an affine $d$-dimensional absolutely irreducible variety over a field $K$. Then the group $\textnormal{Bir}_K(X)$ is approximable (in the sense of Definition \ref{dapprox}) by the family of groups $\{\textnormal{Bir}_\mathbf{F}(X')\}$, where $\mathbf{F}$ ranges over finite fields and $X'$ ranges over $d$-dimensional specifications of $X$ over $\mathbf{F}$ that are absolutely irreducible over $\mathbf{F}$.  
\end{prop}

\begin{proof}
Let $B$ be the $K$-algebra of functions on $X$ and $L$ be its field of fractions, so that $\textnormal{Bir}_K(X)=\textnormal{Aut}_K(L)$.

Suppose that a finite symmetric subset $W$ containing the identity is given in $\textnormal{Aut}_K(L)$. It consists of a finite family $(v_i)$ of pairwise distinct elements of $\textnormal{Aut}_K(L)$. There exists $f\in B-\{0\}$ such that $v_i(B)\subset B[f^{-1}]$ for all $i$. Denote by $u_i:B\to B[f^{-1}]$ the $K$-algebra homomorphism which is the restriction of $v_i$.

Fix generators $t_1,\dots,t_m$ of $B$ as a $K$-algebra, so that $B[f^{-1}]$ is generated by $t_1,\dots,t_m,f^{-1}$ as a $K$-algebra. For each $(i,j)$, we can write $u_i(t_j)$ as a certain polynomial with coefficients in $K$ and $m+1$ indeterminates, evaluated at $(t_1,\dots,t_m,f^{-1})$. Let $C_1$ be the (finite) subset of $K$ consisting of the coefficients of these polynomials ($i,j$ varying). Also, under the mapping $X_j\mapsto t_j$, the $K$-algebra $B$ is the quotient of $K[X_1,\dots,X_m]$ by some ideal; we can  consider a certain finite set of polynomials with coefficients in $K$ generating this ideal. Let $C_2$ be the finite subset of $K$ consisting of the coefficients of those polynomials. Also, $f$ can be written as a polynomial in $t_1,\dots,t_m$; let $C_3\subset K$ consist of the coefficients of this polynomial. Let $A_0$ be the subring of $K$ generated by $C_1\cup C_2\cup C_3$. 

Let $B'_0$ be the $A_0$-subalgebra of $B$ generated by the $t_j$. By generic flatness \cite[Lem.~6.7]{SGA}, there exists $s\in A_0-\{0\}$ such that $B'=B'_0[s^{-1}]$ is flat over $A=A_0[s^{-1}]$. 
Since $A$ contains coefficients of the polynomials defining $B$, we have, in a natural way, $B=B'\otimes_A K$. Moreover, $f\in B'$ and the homomorphisms $u_i$ actually map $B'$ to $B'[f^{-1}]$; if $u'_i$ denotes the corresponding restriction map $B'\to B'[f^{-1}]$, then $u'_i\otimes_A K= u_i$ (here we view $-\otimes_A K$ as a functor). In particular, since the $u_i$ are pairwise distinct by definition, the $u'_i$ are pairwise distinct as well. This means that for all $i\neq i'$ there exists an element $x_{ii'}\in B'$ such that $u_i(x_{ii'})\neq u_{i'}(x_{ii'})$. Let $x\in B'-\{0\}$ be the product of all $u_i(x_{ii'})-u_{i'}(x_{ii'})$, where $\{i,i'\}$ ranges over pairs of distinct indices. Also, fix $k$ large enough so that the element $g=f^k\prod_iu_i(f)\in B'[f^{-1}]-\{0\}$ belongs to $B'-\{0\}$.


There is a natural map $\phi:\textnormal{Spec}(B')\to\textnormal{Spec}(A)$ consisting in taking the intersection with $A$. This map is continuous for the Zariski topology. Consider the open subset of $\textnormal{Spec}(B')$ consisting of those primes not containing $gx$; this is an open subset of $\textnormal{Spec}(B')$ containing $\{0\}$. 
Since $B'$ is $A$-flat, the map $\phi$ is open \cite[Th.~6.6]{SGA}. Therefore there exists $a\in A-\{0\}$ such that every prime of $A$ not containing $a$ is of the form $\mathfrak{P}\cap A$ for some prime $\mathfrak{P}$ of $B'$ not containing $gx$.

Now since $B'$ is $A$-flat and absolutely integral, by \cite[12.1.1]{grothendieck} there exists $a'\in A-\{0\}$ such that for every prime $\mathcal{Q}$ of $A$ not containing $a'$, the quotient ring $B'\otimes_A (A/\mathcal{Q})=B'/\mathcal{Q}B'$ is an absolutely integral $(A/\mathcal{Q})$-algebra.

It follows that if $\mathfrak{m}$ is a maximal ideal of $A$ not containing $aa'$, then $B'/\mathfrak{m}B'$ is an absolutely integral $(A/\mathfrak{m})$-algebra and $\mathfrak{m}B'$ does not contain $gx$. Let us fix such a maximal ideal $\mathfrak{m}\subset A$ (it exists because in a finitely generated domain, the intersection of maximal ideals is trivial, see for instance \cite[Th.~4.19]{Eis}). Since $u'_i$ is a $A$-algebra homomorphism, it sends $\mathfrak{m}B'$ to $\mathfrak{m}B'[f^{-1}]$, and therefore induces a $(A/\mathfrak{m})$-algebra homomorphism $u''_i:B'/\mathfrak{m}B'\to B'[f^{-1}]/\mathfrak{m}B'[f^{-1}]$. Since $x\neq 0$ in $B'/\mathfrak{m}B'$, the $u''_i$ are pairwise distinct.

We need to check that $\dim(B'/\mathfrak{m}B')\le d$.
First, by \cite[Th.\ 13.8]{Eis}, $\dim(B')\le\dim(A)+d$. Now since $B'$ is $A$-flat, by \cite[Th.\ 10.10]{Eis} we have $\dim(B'/\mathfrak{m}B')\le \dim(B')-\dim(A_\mathfrak{m})$. Since $A$ is a finitely generated domain, and $\mathfrak{m}$ is a maximal ideal, we have $\dim(A_\mathfrak{m})=\dim(A)$ (see Lemma \ref{codimfg}), and from the two inequalities above
we deduce $\dim(B'/\mathfrak{m}B')\le d$. (Actually both inequalities are equalities (same references): for the first one, \cite[Th.\ 13.8]{Eis} uses the fact that $A$ is universally catenary, which follows in turn from the fact that $\mathbf{Z}$ is universally catenary, which is part of \cite[Cor.\ 18.10]{Eis}.)

To conclude it is enough to prove the following claim
\begin{cla}
The homomorphisms $u''_i$ uniquely extend to pairwise distinct $(A/\mathfrak{m}$)-automorphisms $v''_i$ of the field of fractions of $B'/\mathfrak{m}B'$ and whenever $v_iv_j=v_k$ we have $v''_iv''_j=v''_k$.
\end{cla}
To check the claim, begin with the following general remark. If $R$ is a domain, $s$ a nonzero element of $R$, and we have two homomorphisms $\alpha,\beta:R\to R[s^{-1}]$, such that $\alpha(s)$ is nonzero, then $\alpha$ uniquely extends to a homomorphism $R[s^{-1}]\to R[(s\alpha(s))^{-1}]$ and we can define the composite map $\alpha\beta:R\to R[(s\alpha(s))^{-1}]$.

Since $g\neq 0$ in $B$, this can be applied to the $K$-algebra homomorphisms $u_i:B\to B[f^{-1}]$, which are given by
\[t_\ell\mapsto u_i(t_\ell)=v_i(t_\ell)=U_{\ell i}(t_1,\dots,t_m)/f^d,\] where $U_{\ell i}\in A[X_1,\dots,X_m]$.
We thus have, for all $\ell$
\begin{align*}
v_i(v_j(t_\ell))= & v_i(U_{\ell j}(t_1,\dots,t_m)/f^d)\\ = & U_{\ell j}(u_i(t_1),\dots,u_i(t_m))/u_i(f)^d\\
= & U_{\ell j}(U_{1i}(t_1,\dots,t_m)/f^d,\dots U_{mi}(t_1,\dots,t_m)/f^d)/u_i(f)^d.
\end{align*}

For all $\ell,j$ can write the formal identity $$U_{\ell j}(X_1/Y,\dots,X_m/Y)Y^\delta=V_{\ell j}(T_1,\dots,T_m,Y)$$ for some $V_{\ell j}\in B[X_1,\dots,X_m,Y]$ and some positive integer $\delta$. Thus $v_iv_j=v_k$ (or equivalently $u_iu_j=u_k$) means that for all $\ell$ we have the equality in $L$

$$U_{\ell j}(U_{1i}(t_1,\dots,t_m)/f^d,\dots U_{mi}(t_1,\dots,t_m)/f^d)/u_i(f)^d=U_{\ell k}(t_1,\dots,t_m)/f^d,$$
that is
$$V_{\ell j}(U_{1i}(t_1,\dots,t_m),\dots U_{mi}(t_1,\dots,t_m))=U_{\ell k}(t_1,\dots,t_m)u_i(f)^df^{d\delta-d},$$
which actually holds in $B'\subset L$. This equality still holds modulo the ideal $\mathfrak{m}B'$. Since $g\neq 0$ in $B'/\mathfrak{m}B'$ (i.e., $f$ and $u_j(f)$ are nonzero elements of the domain $B'/\mathfrak{m}B'$), this equality exactly means that $u''_iu''_j=u''_k$ in the sense above.

Since in particular for every $i$ there exists $\iota$ such that $v_iv_\iota$ and $v_\iota v_i$ are the identity, $u''_iu''_\iota$ and $u''_\iota u''_i$ are the identity; in particular $u''_i$ extends to an automorphism $v''_i$ of the fraction field of $B'/\mathfrak{m}B'$. Since the $u''_i$ are pairwise distinct, so are the $v''_i$. Moreover, whenever $u_iu_j=u_k$, we have $u''_iu''_j=u''_k$ which in turn implies $v''_iv''_j=v''_k$. So the claim is proved, and hence Proposition \ref{redfinig} as well.
\end{proof}

We used the following standard lemma.

\begin{lem}\label{codimfg}
Let $A$ be a finitely generated domain. Then for any maximal ideal $\mathfrak{m}$, we have $\dim(A)=\dim(A_\mathfrak{m})$.
\end{lem}
\begin{proof}
If the characteristic $p$ is positive, $A$ is a finitely generated algebra over the field on $p$ elements, and \cite[Cor.\ 13.4]{Eis} (based on Noether normalization) applies, giving $\dim(A)=\dim(A_\mathfrak{m})+\dim(A/\mathfrak{m})=\dim(A_\mathfrak{m})$.

If the characteristic is zero, we use the fact that the ring $\mathbf{Z}$ is universally catenary \cite[Cor.\ 18.10]{Eis}, to apply \cite[Th.\ 13.8]{Eis}, which yields $\dim(A_\mathfrak{m})=\dim(\mathbf{Z}_{\mathfrak{m}\cap\mathbf{Z}})+\dim(A\otimes_\mathbf{Z}\mathbf{Q})$. Since $\mathfrak{m}$ has finite index, $\mathfrak{m}\cap\mathbf{Z}=p\mathbf{Z}$ for some prime $p$ and $\dim(\mathbf{Z}_{\mathfrak{m}\cap\mathbf{Z}})=1$. So $\dim(A_\mathfrak{m})=1+\dim(A\otimes_\mathbf{Z}\mathbf{Q})$. Since this value does not depend on $\mathfrak{m}$, we deduce that $\dim(A_\mathfrak{m})=\dim(A)$.
\end{proof}

\begin{prop}\label{sofgen}
For every absolutely irreducible affine variety $X$ over a finite field $\mathbf{F}$, the group $\textnormal{Bir}_\mathbf{F}(X)$ is sofic. Actually, its sofic profile is $\preceq n^d$, where $d=\dim(X)$.
\end{prop}
\noindent The proof is similar to the one of Proposition \ref{crefinisofic} and left to the reader. The only additional feature is the fact, which follows from the Lang-Weil estimates (making use of the assumption that $X$ is absolutely irreducible), that for some constants $c>0$ and $c'\in\mathbf{R}$ and every finite extension $\mathbf{F}'$ of $\mathbf{F}$ with $q$ elements, the number of points in $X(\mathbf{F}')$ is $\ge cq^d-c'$.  

From Propositions \ref{redfinig} and \ref{sofgen} we deduce
\begin{cor}\label{mainc}
For every absolutely irreducible affine variety $X$ over a field $K$, the group $\textnormal{Bir}_K(X)$ is sofic. Actually, its sofic profile is $\preceq n^d$, where $d=\dim(X)$.
\end{cor}

\section{Solvability of the word problem}\label{s:wp}

\begin{defn}\label{d:swp}
A countable group has a {\em solvable word problem} if it is finite or isomorphic to $\mathbf{N}$ endowed with a recursive group law, i.e.\ recursive as a map $\mathbf{N}\times\mathbf{N}\to\mathbf{N}$.
\end{defn}

The terminology is motivated by the following elementary characterization in the case of finitely generated groups:

\begin{prop}
A finitely generated group $\Gamma$, given with a surjective homomorphism $p:\mathbb{F}\to \Gamma$ with $\mathbb{F}$ a free group of finite rank, has a solvable word problem if and only if the kernel $N$ of $p$ is a recursive subset of $\mathbb{F}$. (In particular, this does not depend on the choice of $\mathbb{F}$ and the surjective homomorphism $p$.) 
\end{prop}
\begin{proof}
Suppose that $\Gamma$ has solvable word problem in the sense of Definition \ref{d:swp}. We can suppose that $\Gamma=\mathbf{N}$ (or a finite segment therein) with a recursive group law, whose unit is a fixed number $e$. Write $\mathbb{F}=\mathbb{F}_k=\langle t_1,\dots,t_k\rangle$ and set $u_i=p(t_i)$.
If we input any word $w\in \mathbb{F}$, we can compute $w(u_1,\dots,u_k)$ (computed according to the given law on $\mathbf{N}$) and answer yes or no according to whether $w(u_1,\dots,u_k)=e$.

Conversely, suppose that the condition is satisfied. Start from a recursive enumeration $u:\mathbf{N}\to \mathbb{F}$. Given $n$, we define $\kappa(n)=\inf\{k\le n:u(k)u(n)^{-1}\in N\}$. Since $N$ is recursive, $\kappa$ is computable. Note that $\kappa\circ\kappa=\kappa$.
Define $$J=\{n\in\mathbf{N}: \kappa(n)=n\};$$ this is a recursive subset of $\mathbf{N}$. By construction, the composite map $J\stackrel{u}\to \mathbb{F}\to \mathbb{F}/N=\Gamma$ is a bijection. If $J$ is finite, we are done. So suppose $J$ is infinite; then there is a recursive enumeration $q:\mathbf{N}\to J$, defined by an obvious induction. Finally define $n\ast m=q^{-1}(\kappa(u^{-1}(\,u(q(n))u(q(m))\,)))$. This is a recursive law on $\mathbf{N}$ and by construction the composite map $\mathbf{N}\stackrel{q}\to J\stackrel{u}\to \mathbb{F}\to \mathbb{F}/N$ is a magma isomorphism. Thus $\Gamma$ is isomorphic to $(\mathbf{N},\ast)$.
\end{proof}

\begin{thm}\label{fgcsw}
Let $K$ be a field and $n$ a non-negative integer. Then every finitely generated subgroup of $\textnormal{Cr}_d(K)$ has solvable word problem.
\end{thm}

\begin{proof}
Since every finitely generated subgroup of $\textnormal{Cr}_d(K)$ is contained in $\textnormal{Cr}_d$ of a finitely generated field, we can suppose that $K$ is finitely generated. So $K$ is an extension of degree $m$ of some purely transcendental field $L=F(t_1,\dots,t_n)$ with $F$ a prime field ($\mathbf{F}_p$ or $\mathbf{Q}$). Observe that there is an inclusion $\textnormal{Cr}_d(K)\subset \textnormal{Cr}_{md}(L)$, so we can suppose that $K$ itself is a purely transcendental field. We can therefore implement formal calculus of $K$, where in case $F=\mathbf{Q}$, elements of $\mathbf{Q}$ are written as a pair (denominator and numerator) of integers, written in radix~2. 

We can also implement formal calculus on $\textnormal{Cr}_d(K)$. Each element can be written as a $d$-tuple of elements in $K(u_1,\dots,u_d)$; each given as a pair of polynomials (numerator and nonzero denominator). 
The product of two elements in $\textnormal{Cr}_d(K)$ can be computed, namely by composition. That these elements belong to $\textnormal{Cr}_d(K)$ ensures that no zero denominator incurs. Therefore any product can be computed and put in irreducible form. 

The equality of two fractions $P_1/Q_1$ and $P_2/Q_2$ can be checked by computing $P_1Q_2-P_2Q_1$ and checking whether it is the zero polynomial in $F(t_1,\dots,t_n,u_1,\dots, u_d)$. In particular the equality of $(P_1/Q_1,\dots,P_d/Q_d)$ and $(u_d,\dots,u_d)$ can be checked. 
\end{proof}

\begin{rem}
The composition of elements of the Cremona group is submultiplicative for the length of formulas (i.e., the number of symbols involved). It follows that, given fixed Cremona transformations $g_1,\dots,g_k$, the above algorithm, whose input is a group word $w\in F_k$ and whose output is yes or no according to whether $w(g_1,\dots,g_k)=1$ in $\textnormal{Cr}_d(K)$, has 
exponential time with respect to the length of $w$.

The above proof is very similar to that of the more specific case of finitely generated linear groups, due to Rabin \cite{Ra}. However, in the latter case, the elements can be implemented as matrices, and it follows that the algorithm has polynomial time. We do not know whether finitely generated subgroups of the Cremona groups have word problem solvable in polynomial time (however, some of them have no faithful finite-dimensional linear representation).
\end{rem}

\begin{rem}
The above proof shows, more generally, that finitely generated sub-semigroups of the Cremona semigroup (the group of dominant self-maps of the affine space, or equivalently the semigroup of $K$-algebra endomorphisms of the field $K(t_1,\dots,t_n)$) has a solvable word problem, i.e., given $g_1,\dots,g_k$, there is an algorithm whose input is a pair of words $w,w'$ in $k$ letters and the output is yes or no according to whether $w(g_1,\dots,g_k)=w'(g_1,\dots,g_k)$. For the same reason, it has exponential time.
\end{rem}

\end{document}